\documentclass[12pt]{amsart}
\usepackage{amsfonts}
\usepackage{amsmath}
\usepackage{amssymb,latexsym}
\usepackage[mathcal]{eucal}
\usepackage{amscd}
\usepackage[pdftex,bookmarks,colorlinks,breaklinks]{hyperref}

\input xy
\xyoption{all}

\newcommand{\Gcal}{\mathcal{G}}

\newcommand{\Kcal}{\mathcal{K}}

\newcommand{\Ucal}{\mathcal{U}}

\newcommand{\Xcal}{\mathcal{X}}
\newcommand{\Ycal}{\mathcal{Y}}

\newcommand{\ch}{\mathbf{1}}

\newcommand{\Z}{\mathbb{Z}}

\newcommand{\N}{\mathbb{N}}

\newcommand{\Xb}{\mathbf{X}}
\newcommand{\Yb}{\mathbf{Y}}

\newcommand{\al}{\alpha}
\newcommand{\Ga}{\Gamma}
\newcommand{\ga}{\gamma}

\newcommand{\ep}{\epsilon}

\newcommand{\Sig}{\Sigma}

\newcommand{\La}{\Lambda}

\newcommand{\ol}{\overline}

\newcommand{\br}{\vspace{3 mm}}

\newcommand{\nor}{\vartriangleleft}

\newcommand{\tri}{\bigtriangleup}

\newcommand{\Aut}{{\rm{Aut\,}}}

\newcommand{\id}{{\rm{id}}}

\swapnumbers
\theoremstyle{plain}
\newtheorem{thm}{Theorem}[section]

\theoremstyle{definition}

\newtheorem{rem}[thm]{Remark}

\newtheorem{exa}[thm]{Example}

\newenvironment{psmallmatrix}
  {\left(\begin{smallmatrix}}
  {\end{smallmatrix}\right)}

\begin{document}


\title[Distal strongly ergodic actions]
{Distal strongly ergodic actions}

\author{Eli Glasner and Benjamin Weiss}

\address{Department of Mathematics\\
     Tel Aviv University\\
         Tel Aviv\\
         Israel}
\email{glasner@math.tau.ac.il}
\address {Institute of Mathematics\\
 Hebrew University of Jerusalem\\
Jerusalem\\
 Israel}
\email{weiss@math.huji.ac.il}

%


\setcounter{secnumdepth}{2}



\setcounter{section}{0}


\begin{abstract}
Let $\eta$ be an arbitrary countable ordinal. Using results of Bourgain, Gamburd and Sarnak
on compact systems with spectral gap
we show the existence of an action of the free group on three generators $F_3$ on a compact metric space $X$,
admitting an invariant probability measure $\mu$, such that the resulting dynamical system
$(X, \mu, F_3)$ is strongly ergodic and distal of rank $\eta$. In particular this shows that there is a
$F_3$ system which is strongly ergodic but not compact. 
This result answers the open question whether such actions exist.
\end{abstract}
 


\begin{date}
{September 23, 2020}
\end{date}

\maketitle


\setcounter{secnumdepth}{2}



\setcounter{section}{0}

%
%
%
%
%

\section*{Introduction}

In this note we construct a strongly ergodic, distal, non-compact system
for the free group $F_3$ on three generators. 
This answers a question of Ibarluc\'{i}a,  Le Maitre, Tsankov and Tucker-Drob. 
Moreover, we show that for an arbitrary countable ordinal $\eta$, 
there is a strongly ergodic, distal system of rank $\eta$.

The question was brought to our attention 
 in a conversation during the conference entitled ``Structure and Geometry of Polish groups" 
held in Oaxaca, Mexico, June, 2017. 
The first named author thanks BIRS and the organizers for the invitation to 
participate in this very successful conference.
We thank Alex Lubotzky for calling the work \cite{BGS-10} to our attention.

\br

\section{Some preliminaries}

\subsection{Structure theory}

Let $\Ga$ be a discrete countable infinite group.
A {\em $\Ga$-dynamical system} is a quadruple $\Xb = (X, \mathcal{X}, \mu, T)$,
where $(X,\mathcal{X},\mu)$ is a standard probability space
and $\ga \mapsto T_\ga$ is a homomorphism from $\Ga$ into the 
Polish group $\Aut(X,\mu)$ of invertible measure preserving transformation of
$(X, \mathcal{X}, \mu)$. 
When there is no room for confusion we write $\ga x$ instead of $T_\ga x$.
When $\mathbf{X}$ and $\mathbf{Y} = (Y. \mathcal{Y}, \nu, S)$ are two dynamical systems, we say that
$\mathbf{Y}$ is a {\em factor} of $\mathbf{X}$ ( or that 
$\mathbf{X}$ is an {\em extension} of $\mathbf{Y}$) is there is 
a measurable map $\pi : X \to Y$ such that $\pi_*(\mu)=\nu$ and such that
$\pi(T_\ga x) = S_\ga\pi(x)$ for every $\ga \in \Ga$ and $\mu$ almost every $x \in X$.
The map $\pi$ is called a {\em factor map} (or an {\em extension}).

The system $\mathbf{X}$ is {\em ergodic} if every $\Ga$-invariant set $A \in \mathcal{X}$
(i.e. $T_\ga A = A \pmod \mu$ for every $\ga \in \Ga$)  is trivial (i.e. $\mu(A)(1 -  \mu(A))=0$).

Let $\Yb$ be a dynamical system and $(U,\Ucal,\rho)$
a standard probability space. Let $\al: \Ga\times Y\to
{\Aut}(U,\rho)$ be a measurable cocycle;   that is
$\al$ satisfies the {\em cocycle equation}
$$
\al(\ga\ga',y)=\al(\ga,\ga' y)\al(\ga',y).
$$
We define the {\em skew-product system\/}\label{def-skew-prod}
$\Yb  \times_\al (U,\rho)$ to be the system
$(Y\times U,\Ycal\otimes\Ucal,\mu\times\rho,\Ga)$, where
$\ga(y,u)=(\ga y,\al(\ga,y)u)$. (Check that this indeed
defines an action of $\Ga$ on $X=Y\times U$.)
When $\Ga=\Z$ with the measure-preserving invertible map
$T$ as the generator of the $\Z$-system, a cocycle $\al$
is completely determined by the map $\al(y)=\al(T,y)$ and we have
$$
\al(n,y)=\begin{cases}
\al(T^{n-1} y)\cdots\al(Ty)\al(y) & {\text{for}}\  n\ge 1\\
\id & {\text{for}}\  n=0\\
\al(T^n y)^{-1}\cdots\al(T^{-1} y)^{-1} & {\text{for}}\  n< 0.
\end{cases}
$$
In the special case where $U$ is a compact group and $\rho$ is its normalized Haar measure,
any measurable function $\al : Y \to U$ defines a skew product by the formula:
$$
T(y,u) = (Ty, \al(y)u), \ y \in Y, u \in U.
$$

We have the following basic theorem:

\begin{thm}[Rohlin]\label{rsp}
Let $\Xb\to\Yb$ be a factor map of dynamical systems with
$\Xb$ ergodic, then $\Xb$ is isomorphic to a skew-product
over $\Yb$. Explicitly, there exist a standard probability space
$(U,\Ucal,\rho)$ and  a measurable cocycle $\al: \Ga\times Y\to
{\Aut}(U,\rho)$ with $\Xb\cong \Yb \times_\al (U,\rho)
=(Y\times U,\Ycal\otimes\Ucal,\nu\times\rho,\Ga)$, where
$\ga(y,u)=(\ga y,\al(\ga,y)u)$.
\end{thm}

The topology on $\Aut(X,\mu)$ is induced by a complete metric
$$
D(S,T) = \sum_{n \in \N} 2^{-n}  (\mu(SA_n \tri TA_n) + \mu(S^{-1}A_n \tri T^{-1}A_n)), 
$$
with $\{A_n\}_{n \in \N}$ a dense sequence in the measure algebra
$(\mathcal{X},d_\mu)$, where $d_\mu(A,B) = \mu(A \tri B)$.
Equipped with this topology $\Aut(X, \mu)$ is a Polish topological group and we say that
the dynamical system $\mathbf{X}$ is {\em compact} if the 
image $\{T_{\ga} : \ga \in \Ga\}$ is a precompact subgroup of $\Aut(X,\mu)$.

\begin{exa}
Let $\Ga = F_2$, the free group of rank $2$. Let
$$
X = \lim_{\leftarrow} \{\Ga/ N : N \nor \Ga, \ {\text{with}}\ [\Ga:N] < \infty\}.
$$
This is the {\em profinite completion} of $\Ga$. It is a compact metrizable topological group
and thus admits a unique normalized Haar measure $\mu$.
There is a canonical embedding $\phi : \Ga \to X$ with a dense image
and for $\ga \in \Ga$ we let $T_\ga x = \phi(\ga)x, \ x \in X$. With $\mathcal{X}$ the 
algebra of Borel subsets of $X$, $\Xb = (X, \mathcal{X}, \mu, T)$ is an 
ergodic compact $F_2$ dynamical system.
\end{exa}

It turns out that, in fact, every ergodic compact $\Ga$ system $\mathbf{X}$ has the following form
$X = K/H$, where $K$ is a compact metrizable topological group, $H <K$ is a closed subgroup,
$\mu$ is the induced Haar measure on $X$, and the action of $\Ga$ on $X$ is via a 
homomorphism $\phi : \Ga \to K$ with dense image so that
$T_\ga kH = \phi(\ga)kH,\ k \in K$.

The notion of compactness can now be relativized as follows:

\br

An extension $\pi : \Xb \to \Yb$ is a {\em compact extension} if there is a compact metrizable group $K$,
a closed subgroup $H < K$ and a cocycle $\al : \Ga \times Y  \to K$ such that
$$
\Xb\cong \Yb \times_\al (K/H,\rho)
=(Y\times K/H,\Ycal\otimes\Kcal,\nu\times\rho,\Ga),
$$
where $\rho$ is the Haar measure on $K/H$ and for each $\ga \in \Ga$,
$\ga(y, kH)=(\ga y,\al(\ga,y)kH)$.

\br

This construction can be iterated and a dynamical system $\mathbf{X}$ is called {\em distal} 
if it is an iteration of countably many compact extensions, where in the
(possibly transfinite construction) at a limit ordinal one takes an inverse limit.
The so called {\em distal tower} is unique if at each stage one takes the maximal compact extension
(within $\Xb$).
The height of this tower (a countable ordinal) is called the {\em rank} of the distal system $\Xb$,
so the height of a  compact action is 1.

\br

An extension of dynamical systems $\pi: \Xb \to \Yb$ is called a {\em {weakly mixing extension}}
when the corresponding relative product
$(X \underset Y \times X, \mu \underset \nu \times \mu, \Ga)$ is ergodic.
In particular $\Xb$ is {\em {weakly mixing}} when the product system $\Xb \times \Xb$
is ergodic. 

We now can state the following

\begin{thm}[Furstenberg-Zimmer structure theorem]
Every ergodic system $\Xb$ has (uniquely) a largest distal factor $\pi : \Xb \to \Yb$
and the extension $\pi$ is a weakly mixing one.
\end{thm}

In \cite{BF} Beleznay and Foreman show that for $\Ga=\mathbb{Z}$, for every countable ordinal 
$\eta$ there is an ergodic distal system of height $\eta$.
%

We refer e.g. to \cite{G-03} for more details on structure theory in ergodic theory. 
\br

\subsection{Strong ergodicity, Kazhdan's property and expanders}


Recall that a probability measure preserving action $(X, \mathcal{X}, \mu, G)$ 
is called {\em strongly ergodic} if there is no sequence of sets $A_n \in \mathcal{X}$
with $\mu(A_n)=1/2, \forall n \in \N$ such that
$$
\lim \mu(g A_n \tri A_n) =0, \quad \forall g \in G.
$$
(Informally,  we say that $\Xb$ admits no {\em almost invariant sets}.)

\br

Theorems of Connes and Weiss \cite{CW}, Connes, Feldman and Weiss  \cite{CFW} 
and  Schmidt \cite[Theorems 2.4 and 2.5]{Sch}, asserts that:

\br

(i) The group $G$ is amenable iff every nontrivial ergodic $G$-system is not strongly ergodic, \cite{CFW}
and \cite{Sch}.

(ii) The group $G$ has Kazhdan's property T iff every ergodic $G$-system is strongly ergodic, \cite{CW}.

\br

The {\em Cheeger constant} of a finite $k$-regular graph $\Gcal$ is defined as
$$
h(\Gcal) = \min \left \{\frac{|\partial A|}{|A|} : A \subset V(\Gcal), \ 0 < |A| \leq \frac{|V(\Gcal)|}{2} \right \}.
$$
The graph $\Gcal$ is an {\em $\ep$-expander} if $h(\Gcal) > \ep$.

A family $\{G_i\}$ of finite groups is called an {\em expander family} if for some $k$ and $\ep>0$
there are generating sets $\Sig_i$ for $G_i$ with $|\Sig_i| \leq k$, such that all the Cayley graphs
$\Gcal(G_i, \Sig_i)$ are $\ep$-expanders.

\begin{exa}\label{exa}
Let $S$ be a set of elements in $SL_2(\Z)$. 
If $\langle S \rangle$, the group generated by $S$, is a finite index subgroup of 
$SL_2(\Z)$, Selberg's theorem \cite{S-65} implies (see e.g. \cite{L-94} Th. 4.3.2 
and Ex. 4.3.3.D) 
that $\mathcal{G}(SL_2(\Z_p), S_p)$, the Cayley graphs of $SL_2(\Z_p) = SL_2(\Z/p\Z)$ with respect 
to $S_p$, the natural projection of $S$ modulo $p$, form a family of expanders as $p\to \infty$.
For example we can take
$$
S = \left\{ 
\begin{pmatrix}
1 & 1 \\
0 & 1
\end{pmatrix},  
\begin{pmatrix}
1 & 0 \\
1 & 1
\end{pmatrix}  
\right\}.
$$
\end{exa}

\begin{rem}
Suppose $G$ is a finite group generated by $S \subset G$, such that $h(\Gcal(G,S)) > \ep$.
Then for $A \subset G$ with $|A| \approx \frac{|G|}{2}$ we have $\partial A \geq \ep |A|$ and
$|SA \tri A| \geq \ep |A|$.
We can interpret this as the absence of almost invariant subsets.
\end{rem}
%
%
%
%
%

In our construction we will have a finite set 
$S \subset SL_2(\Z)$ that generates a Zariski dense subgroup $\La$ and we will consider the
diagonal action of $\La$ on products of the form
$$
K_P = \prod  \{SL_2(\Z_p) : p \in P\}
$$
taken over an infinite set of primes $P$.
In order to ensure that this action is strongly ergodic we need to know that finite products
$$
K_F = \prod  \{(SL_2(\Z_p), S_p) : p \in F\},
$$
where $F \subset P$ is finite and $S_p$ is the image of $S$ in $SL_2(\Z_p)$,
form an expander family.
In our situation $S_p$ in the product will generate $K_F$ and since such a finite product
is isomorphic to $SL_2(\Z_q)$, where $q = \prod \{p : p \in F\}$, we will need to know that
the family $\{SL_2(\Z_q), S_q)\}$ is a family of expanders.
This latter fact is established in the following theorem of  Bourgain, Gamburd and Sarnak
 \cite[Theorem 1.2]{BGS-10}.

%
%

\begin{thm}\label{BGS}
Let  $\Lambda$ be a Zariski dense subgroup of $SL_2(\Z)$ 
and let $S$ be a finite symmetric set of generators for $\Lambda$. 
Then for $q$ square-free the family of Cayley graphs $\mathcal{G}( \Lambda/ \Lambda(q),S)$ 
is an expander family.
\end{thm}

\br

In a sharp contrast to the result of Belezney and Foreman \cite{BF} mentioned above,
Chifan and Peterson (unpublished) and, independently, Ibarluc\'{i}a and Tsankov,
\cite{IT}, show that for $\Ga$ with Kazhdan's property T, every distal system is already compact.
(In fact, they show that every distal extension of an ergodic system is a compact extension.)


This raises the question whether a group $G$ which is neither amenable nor a Kazhdan group
can admit a strongly ergodic distal action which is not compact.
In the next section we will answer this question positively for the free group  on three generators $G = F_3$.


%
%

\br

\section{An example of a strongly ergodic, distal but not compact $F_3$ dynamical system}

Our goal is to construct an $F_3$ action which is strongly ergodic, distal but not compact.

%
%
%
%
%

We want to find compact metrizable groups $K$ and $L$ with the following properties:
\begin{itemize}
\item
$K$ contains three elements $a, b$ and $c$ such that the subgroup
$\La= \langle a, b, c \rangle$ generated by them is free and dense in $K$,
$\ol{\La} = K$. 
\item
The subgroup  $\La_0 = \langle a, b \rangle < \La$ is also dense in $K$,
$K= \ol{\La_0}$. 
\item
$L$ contains two elements $f, g$ such that the subgroup
$ \Sig =\langle f, g \rangle$ generated by them is free and dense in $L$,
$L =  \ol{\Sig}$. 
\end{itemize}
Moreover,
denoting by $F_2, F_3$ the free groups on two and three generators respectively, 
all of the following actions (under left multiplication
and with respect to the corresponding Haar measures)
are strongly ergodic:
\begin{enumerate}
\item
the action of $F_3$ via $\La$ on $K$,
\item
the action of $F_2$ via $\La_0$ on $K$,
\item
the action of $F_2$ via $ \Sig$ on $L$,
\item
the diagonal action of $F_2$ via $\langle (a,f), (b,g) \rangle$ on $K \times L$.
\end{enumerate}
(Of course conditions (1) - (3) follow from condition (4).)

\br

Here is a construction of such groups.

Let 
$x=\begin{psmallmatrix}1 & 2\\0 & 1\end{psmallmatrix}$
and 
$y = \begin{psmallmatrix}1 & 0\\2 & 1\end{psmallmatrix}$ 
be two elementary 
matrices generating a free group, which we 
call $H$, in $SL_2(\Z)$ (it is actually of finite index).
Let 
$a = x^2 = \begin{psmallmatrix}1 & 0\\4 & 1\end{psmallmatrix}$ 
and 
$b = y^2 = \begin{psmallmatrix}1 & 4\\0 & 1\end{psmallmatrix}$, 
and let $c$ be another element in 
$H$ which together with $a$ and $b$ generate a free group $\La$ on $3$ generators 
(such an element clearly exists as a and b generate a subgroup of infinite index in $H$).
Now set
$$
K = \prod \{\La_p : p\ {\text{prime, \ and}}\ \ p \equiv1\pmod{4}\}
$$
and
$$
L = \prod \{\La_p : p\ {\text{prime, \ and}}\ \ p \equiv3\pmod{4}\},
$$
where $\La_p$ is the image of $\La$ in $SL_2(\Z_p)$.

Note that, by the Chinese remainder theorem,
as  $a = x^2 = \begin{psmallmatrix}1 & 0\\4 & 1\end{psmallmatrix}$
and $b = y^2 = \begin{psmallmatrix}1 & 4\\0 & 1\end{psmallmatrix}$,
we have that, for any finite set of distinct primes $p_1,p_2,\dots ,p_t$ ($p_i \not = 2$),
denoting $q = p_1p_2\cdots p_t$,
the image of $\La_0$ in $SL_2(\Z_q)$ is all of $SL_2(\Z_q)$.

The strong ergodicity of $F_2$ via $\La_0$ on $K$ follows from Theorem \ref{BGS}.
Indeed if $\ch_A$ is the indicator function of an almost invariant set $A$ with measure $\approx 1/2$
in $K$, then its projection onto a sufficiently large finite product will give rise to an almost invariant
nontrivial function, contrary to the fact that the family is a family of expanders.

Of course the arguments for the actions on $L$ and $K \times L$ are analogous.

\br

Next let $K_1 =  \ol{\langle c \rangle} < K$. 
By \cite{Z-77} (see also \cite{GW}) there is a cocycle $\phi_0 : K_1 \to L$ such that the corresponding $\Z$-action
on $K_1 \times L$ given by
$$
T_c(x,y) = (cx, \phi_0(x) y), \quad x \in K_1, y \in L,
$$ 
is ergodic.
Note that, as $L$ is non-commutative this distal $\Z$-action, is necessarily not compact.
To see this recall that $T_c$ is ergodic and if it would be compact, as is well known,
its centralizer would be commutative. In fact, ergodic compact $\Z$-actions are rotations of
compact monothetic groups and the centralizer is the compact group itself.

Define $\phi : K \to L$ by the formula
$$
\phi(x) = \phi_0(k_t^{-1}x) \quad {\text {for $x$ in the coset $k_tK_1$}},
$$
where $t \mapsto k_t,  K/K_1 \to K$ is a Borel section for the map $K \to K/K_1$.

On $X \times Y := K \times L$ define an $F_3$ action $\{T_t\}_{t \in F_3} : K \times L \to K\times L$ by:
\begin{gather*}
T_a(x,y) = (ax, fy)\\
T_b(x,y) = (bx, gy)\\
T_c(x,y) = (cx, \phi(x)y).
\end{gather*}
%

\br

With this data at hand we can now prove our main result:

\begin{thm}\label{dnc}
The $T$ action of $F_3$ on $K \times L$ is distal, not compact, and strongly ergodic.
\end{thm}

\begin{proof}
%
Clearly the $T$ action on $K \times L$ is distal (of order $2$). It can not be 
compact because the restriction
to the $\Z$-action $T_c : K_1 \times L \to K_1 \times L$ is not compact.
Finally, by construction, the $F_2$-action on $K \times L$ is strongly ergodic and, 
a fortiori, so is the $T$-action. Our proof is complete.
\end{proof}

\br

\section{Distal strongly ergodic $F_3$-systems of arbitrary countable rank}

With only some minor modifications the proof of Theorem \ref{dnc} can be applied to prove the
following stronger result.

\begin{thm}\label{rank-dnc}
Let $\eta$ be an arbitrary countable ordinal $\eta$. 
Then there exists an action of the free group on three generators $F_3$ on a compact metric space $X$,
admitting an invariant probability measure $\mu$, such that the resulting dynamical system
$(X, \mu, F_3)$ is strongly ergodic and distal of rank $\eta$.
\end{thm}

\begin{proof}
We keep the notations introduced in the previous section with the following modifications:
Let $x=\begin{psmallmatrix}1 & 2\\0 & 1\end{psmallmatrix}$
and $y = \begin{psmallmatrix}1 & 0\\2 & 1\end{psmallmatrix}$
and let $H$ be the finite index subgroup $\langle x, y \rangle$. 
Let 
$a = x^2 = \begin{psmallmatrix}1 & 0\\4 & 1\end{psmallmatrix}$ 
and 
$b = y^2 = \begin{psmallmatrix}1 & 4\\0 & 1\end{psmallmatrix}$, 
and let $c$ be another element in 
$H$ which together with $a$ and $b$ generate a free group $\La$ on $3$ generators. 
Recall that for sufficiently large $p$, the image of $S_0 = \{a, b\}$ (and a fortiori that of
the set $S = \{a, b, c\}$) in $SL_2(\Z_p) \cong SL_2(\Z)/ SL_2(p\Z)$
generates $SL_2(\Z_p)$.
Let $P$ be the set of primes for which this is true.

Let $P = \bigcup_{j =0}^\infty P_j$ be some partition of $P$, with each cell $P_j$ being infinite.
Choose an arbitrary bijection 
$$
\Phi : \{1,2,\dots\} \longleftrightarrow 
\{\al < \eta : \al \ {\text{is a successor ordinal}}\}
$$
and set
$$
K = \prod \{\La_p : p\in P_0\},
$$
and
$$
K_{\Phi(j)} = \prod \{\La_p : p\in P_j\},
$$
where $\La_p$ is the image of $\La$ in $SL_2(\Z_p)$.

We now consider the action of $S$, induced by left multiplication, on $K$ and on each $K_\al$.
As above all of these actions are strongly ergodic, and so is
the action on the product space
\begin{equation}\label{X}
X = K \times L = K \times \prod_{\al} K_\al.
\end{equation}

We now recall the following result from \cite{GW} (adapted to our needs here) :
\begin{thm}\label{thm-main}
Given an arbitrary countable ordinal $\eta$ and a transfinite sequence 
$$
\{G_\al : \al =0 \ {\text{and}}\  \al < \eta,\ \al  {\text{ a successor ordinal}} \},
$$
where for each $\al$ $G_\al$ is an infinite compact second countable topological group,
with the only requirement that $G_0$ be an infinite monothetic compact group $(G_0,c)$,
there exists an ergodic distal system $\Xb = (X, \Xcal, \mu, A)$ of rank $\eta$ such that, 
in its canonical distal tower,
for each successor $\al$ the extension $\Xb_\al \to \Xb_{\al -1}$ is a $G_{\al-1}$-extension,
and such that the induced action of $A$ on $G_0$ is via multiplication by $c$.
\end{thm}

We let  $K_0 =  \ol{\langle c \rangle} < K$, and we consider $(K_0,c)$
as a (compact) $\Z$-dynamical system.
Applying Theorem \ref{thm-main}, with $G_\al = K_\al$ for all $0 \leq \al < \eta$,
we can now realise an action of the group $\langle c \rangle$
on the space $K_0 \times L$ by $T_c(x) = Ax$ (where for a limit ordinal $\xi \leq \eta$
the corresponding system is determined as the inverse limit of the preceding systems 
directed by $\{\al : \al < \xi\}$).

Next we extend the $T_c$ action to the space $X= K \times L$ in (\ref{X}) 
by defining for $x = (k,l) \in K_0 \times L$,
$$
T_c(k, l) = A(k_t^{-1}k,l), \quad {\text{for $k$ in the coset $k_t K_0$}},
$$
where $t \mapsto k_t,  \ K/K_0 \to K_0$ is a Borel section for the map $K_0 \to K/K_0$.

Finally on $K \times L$ define an $F_3$ action $\{T_t\}_{t \in F_3} : K \times L \to K\times L$ by:
\begin{gather*}
T_a(k,l) = (ak, al)\\
T_b(k,l) = (bk, gl)\\
{\text{and $T_c(k,l)$ as above}}.
\end{gather*}
As in the previous section this completes our proof.
\end{proof}


\end{document}